\documentclass[12pt,oneside]{article}

\usepackage[english]{babel}
\usepackage[letterpaper,top=2cm,bottom=2cm,left=3cm,right=3cm,marginparwidth=1.75cm]{geometry}
\usepackage{amsmath,amssymb,amsthm,mathtools}
\usepackage{times}
\usepackage{indentfirst}
\usepackage[colorlinks=true,allcolors=blue]{hyperref}

\newtheorem{theorem}{Theorem}
\newtheorem{lemma}{Lemma}
\newtheorem{proposition}{Proposition}
\newtheorem{corollary}{Corollary}
\theoremstyle{definition}
\newtheorem{definition}{Definition}
\newtheorem{remark}{Remark}
\newtheorem{example}{Example}
\newtheorem{problem}{Problem}
\newcommand{\F}{\mathbb{F}}

\newcommand{\rank}{\operatorname{rank}}
\newcommand{\nullity}{\operatorname{nullity}}
\newcommand{\sfp}{\operatorname{sfp}}

\newcommand{\Ker}{\operatorname{Ker}}

\newcommand{\diag}{\operatorname{diag}}

\begin{document}

\title{A New Sufficient Condition for Oriented Graphs Determined by Their Generalized Skew Spectra}
\author{\small Limeng Lin$^{{\rm a}}$\quad\quad Wei Wang$^{{\rm b}}$\quad\quad Wei Wang$^{{\rm a}}$\thanks{Corresponding author: wang\_weiw@xjtu.edu.cn}\quad\quad Hao Zhang$^{{\rm c}}$
\\
{\footnotesize $^{{\rm a}}$School of Mathematics and Statistics, Xi'an Jiaotong University, Xi'an 710049, P. R. China}
\\
{\footnotesize $^{{\rm b}}$School of Mathematics, Physics and Finance, Anhui Polytechnic University, Wuhu 241000, P. R. China}
\\
{\footnotesize $^{{\rm c}}$School of Mathematics, Hunan University, Changsha 410082, P. R. China}}
\date{}
\maketitle

\begin{abstract}
Characterizing graphs uniquely determined by their spectra (DS) is a core open problem in spectral graph theory. While this problem has been extensively investigated for simple undirected graphs, it remains relatively underexplored for oriented graphs. For a simple undirected graph $G$ equipped with an orientation $\sigma$, the corresponding oriented graph $\Sigma=(G,\sigma)$ is the digraph obtained by orienting each edge of $G$ according to $\sigma$. An oriented graph $\Sigma$ is said to be \emph{determined by its generalized skew spectrum} (DGSS) if every oriented graph sharing the same generalized skew spectrum is isomorphic to $\Sigma$.

This paper develops a new sufficient criterion for recognizing DGSS controllable oriented graphs, which applies to a much broader family of graphs than previously known results. Let $S$ be the skew-adjacency matrix of $\Sigma$, $W(\Sigma)=[e,Se,\ldots,S^{n-1}e]$, and $d_n$ the last invariant factor of $W(\Sigma)$. For each odd prime $p$, we define the polynomial $\Phi_p(\Sigma;x)=\gcd(\chi(S;x),\chi(S+J;x))$ over the finite field $\mathbb{F}_p$, which is invariant under generalized skew cospectrality.

By analyzing the square-free part of $\Phi_p(\Sigma;x)$ and the associated $p$-main polynomial, we establish a DGSS sufficient condition under the square-free assumption on $d_n$. The proposed criterion allows higher $p$-nullity and recovers the square-free determinant criterion of Qiu, Wang and Wang~(2019) as a special case. We further provide illustrative examples to verify the wider applicability of our new condition and to highlight the role of the compatibility constraints on the irreducible factors of $\Phi_p(\Sigma;x)$.


\end{abstract}



\section{Introduction}
Let $G$ be a simple graph with vertex set $V(G)=\{v_1,v_2,\ldots,v_n\}$ and edge set $E(G)$. An \emph{oriented graph} $\Sigma=(G,\sigma)$ is obtained from $G$ by assigning an orientation to each edge according to $\sigma$. The graph $G$ is called the \emph{underlying graph} of $\Sigma$.

  The skew-adjacency matrix of $\Sigma$ is the matrix $S=S(\Sigma)=(s_{ij})$ defined by
\begin{equation}\label{eq:skew-adjacency}
 s_{ij}=\begin{cases}
 1, & \text{if }(v_i,v_j)\text{ is an arc of }\Sigma,\\
 -1, & \text{if }(v_j,v_i)\text{ is an arc of }\Sigma,\\
 0, & \text{otherwise}.
 \end{cases}
\end{equation}
Thus $S^{\rm T}=-S$. The \emph{skew spectrum} of $\Sigma$ consists of all the eigenvalues (including the multiplicities) of $S(\Sigma)$; see \cite{CaversEtAl,ShaderSo}. We write $\chi(S;x)=\det(xI-S)$ for the characteristic polynomial of $S$. 
Let $J$ denote the all-one matrix of order $n$. Two oriented graphs $\Sigma$ and $\Delta$ are said to be  \emph{generalized skew cospectral} if $
 \chi(S(\Sigma);x)=\chi(S(\Delta);x)$ and
 $\chi(S(\Sigma)+J;x)=\chi(S(\Delta)+J;x).$
An oriented graph $\Sigma$ is said to be \emph{determined by its generalized skew spectrum} ($\mathrm{DGSS})$ if every oriented graph with the same generalized skew spectrum is isomorphic to $\Sigma$.

An oriented graph $\Sigma$ is called \emph{self-converse} if it is isomorphic to its converse $\Sigma^{\rm T}$, obtained from $\Sigma$ by reversing the direction of every arc. Since $S(\Sigma^{\rm T})=-S(\Sigma)$, the oriented graphs $\Sigma$ and $\Sigma^{\rm T}$ have the same generalized skew spectrum. It follows that every $\mathrm{DGSS}$ oriented graph must be self-converse.

The \emph{skew-walk matrix} of $\Sigma$ is defined by
\[
 W=W(\Sigma):=\bigl[e,Se,\ldots,S^{n-1}e\bigr],
\]
where $e$ is the all-one vector.  We call $\Sigma$ \emph{controllable} if  $W(\Sigma)$ is nonsingular.

A fundamental problem in spectral graph theory is to characterize which graphs are
determined by their spectra. In general, proving that a graph is $\mathrm{DS}$
is difficult, and only a limited number of graph families are known to possess
this property. We refer to \cite{vanDamHaemers1,vanDamHaemers2} and the
references therein for surveys of related results.

To obtain stronger spectral
characterizations, Wang and Xu~\cite{WangXu} studied graph determination
from the perspective of the generalized spectrum. The generalized spectral determination problem has since been extensively
studied for both graphs and oriented graphs. For simple graphs, Wang~\cite{Wang 2013,Wang2017} established
criteria based on the arithmetic structure of the walk matrix, including
square-freeness conditions on its determinant. This approach was subsequently
extended to several related settings, including generalized $Q$-spectral
characterization, almost controllable graphs, and signed trees; see
\cite{Qiu Q,almost,ref5}. For oriented graphs, Qiu et al.~\cite{QiuWangWang}
obtained an analogous criterion. Let $\mathcal{F}_n$ denote the set of all
controllable oriented graphs on $n$ vertices.

\begin{theorem}[\cite{QiuWangWang}]\label{thm:old}
Let $\Sigma\in \mathcal{F}_n$ be self-converse. If $
2^{-\lfloor n/2\rfloor}\det W(\Sigma),$
which is always an integer, is odd and square-free, then $\Sigma$ is $\mathrm{DGSS}$.
\end{theorem}

The condition in Theorem~\ref{thm:old} can be expressed in terms of the Smith normal form. Recall that every nonsingular integral matrix $M$ is integrally equivalent to a unique diagonal matrix $\diag(d_1,d_2,\ldots,d_n)$ with positive entries satisfying $d_i\mid d_{i+1}$ for $1\leq i<n$. This diagonal matrix is called the \emph{Smith normal form} of $M$, and $d_1,d_2,\ldots,d_n$ are its invariant factors. 

Let $p$ be a prime. We write $\mathbb{F}_p$ for the finite field of order $p$ and
$\overline{\mathbb{F}}_p$ for its algebraic closure. For an oriented graph
$\Sigma$ with skew-adjacency matrix $S$, the \emph{invariant polynomial} of
$\Sigma$ over $\mathbb{F}_p$ is defined by
\[
\Phi_p(\Sigma;x)
:=
\gcd\bigl(\chi(S;x),\chi(S+J;x)\bigr)
\in \mathbb{F}_p[x].
\]
The polynomial $\Phi_p(\Sigma;x)$ is invariant under generalized skew
cospectrality; see Section~\ref{proof thm} for more details.

Let $f\in\mathbb{F}_p[x]$ be a monic polynomial with irreducible factorization $
f=\prod_{i=1}^r f_i^{a_i},$
where $f_1,\ldots,f_r$ are distinct monic irreducible polynomials and
$a_1,\ldots,a_r$ are positive integers. The \emph{square-free part} of $f$ is
defined by $
\sfp(f):=\prod_{i=1}^r f_i$. For an integral matrix $M$, we denote by $\rank_p M$ and
$\nullity_p M$ its rank and nullity over $\mathbb{F}_p$, respectively.

Recently, Wang et al.~\cite{WangWangZhu} improved the method of
\cite{Wang 2013} by extending it to the case $\rank_p W(G)\leq n-1$, for every prime $p$ of the last invariant factor $d_n (W(G))$.
Their approach uses the Smith normal form of the walk matrix and a new
invariant polynomial. Their result is stated as follows.

\begin{theorem}[\cite{WangWangZhu}]\label{wang}
Let $G$ be a controllable simple graph, and let $d_n$ be the last invariant
factor of $W=W(G)$. Suppose that $d_n$ is square-free. If for each odd prime factor $p$ of $d_n$,
\begin{equation}
\deg\operatorname{sfp}(\Phi_p(G; x)) =\operatorname{nullity}_p W,
\end{equation}
then $G$ is $\mathrm{DGS}$.
\end{theorem}

In this paper, we extend Theorem~\ref{wang} to oriented graphs. However, this extension is nontrivial. In the skew-adjacency setting, the degree-nullity condition alone is not sufficient, since new cospectral mates may arise when opposite roots belong to different irreducible factors. This motivates an additional condition concerning the interaction of irreducible factors under $x\mapsto -x$.
Moreover, we provide a more concise proof of the characteristic polynomial factorization over invariant subspaces, which applies to a more general class of matrices.
 We now state our main result.

\begin{theorem}\label{thm:main}
Let $\Sigma\in \mathcal{F}_n$, and let $d_n$ be
the last invariant factor of $W(\Sigma)$. Suppose that $d_n$ is square-free and $
 \rank_2 W(\Sigma)=\left\lceil\frac n2\right\rceil.$
For every odd prime divisor $p$ of $d_n$, assume that the following conditions
hold:

\noindent\textup{(i)}
\[
\deg\sfp(\Phi_p(\Sigma;x))=\nullity_p W(\Sigma);
\]

\noindent\textup{(ii)} write
\[
\sfp(\Phi_p(\Sigma;x))
=
\varphi_1(x)\cdots\varphi_s(x)
\]
as a product of distinct monic irreducible polynomials over $\mathbb{F}_p$. Then $
\gcd\bigl(\varphi_i(x),\varphi_j(-x)\bigr)=1$,
whenever $i\neq j$.

Then $\Sigma$ is $\mathrm{DGSS}$.
\end{theorem}

\begin{remark}\label{rem:vacuous}
Condition~\textup{(ii)} is automatically satisfied when $s=1$. In particular,
it imposes no additional restriction when $\deg\sfp(\Phi_p(\Sigma;x))=1.$ Moreover, condition~\textup{(ii)} cannot in general be omitted from the present criterion, as illustrated by Example~\ref{exam 2} in Section~\ref{sec:examples}.
\end{remark}
Indeed, Theorem~\ref{thm:old} essentially corresponds to the case
$\nullity_pW(\Sigma)=1$ for every odd prime divisor $p$ of the last invariant
factor of $W(\Sigma)$. In contrast, Theorem~\ref{thm:main} allows higher $p$-nullity and can therefore
identify a strictly broader class of controllable oriented graphs as
$\mathrm{DGSS}$. The additional condition~\textup{(ii)} is introduced to handle the
interaction of irreducible factors under the transformation $x\mapsto -x$,
which is specific to the skew-adjacency setting. Moreover,
Theorem~\ref{thm:old} is recovered as a special case of
Theorem~\ref{thm:main}. Further details are given in
Section~\ref{proof thm}.

Accordingly, self-converseness was imposed as an explicit assumption in Theorem~\ref{thm:old}. One aim of the present paper is to obtain a criterion from which self-converseness follows as a consequence rather than an assumption.

The paper is organized as follows. Section~\ref{sec:preliminaries} recalls regular rational orthogonal matrices and introduces the prime-exclusion method and the required polynomial decomposition.  The proof of Theorem~\ref{thm:main} is completed in Section~\ref{proof thm}. Section~\ref{sec:examples}
presents examples illustrating the scope of the new criterion and the role of condition~\textup{(ii)}. Conclusions are given in Section~\ref{sec:conclusion}.

\section{Preliminaries}\label{sec:preliminaries}

In this section, we present several preliminary results that will be used later. Throughout the section, we simply write $S=S(\Sigma)$ and $W=W(\Sigma)$. Let $p$ be a prime. For a nonzero integer $a$, let $
v_p(a)=\max\bigl\{k\geq 0 : p^k\mid a\bigr\}.$
\subsection{Regular rational orthogonal matrices}\label{rational orthogonal matrices}
In this subsection, we outline the main strategy for proving that oriented graphs are $\mathrm{DGSS}$. Recall that a rational orthogonal matrix $Q$ is called \emph{regular} if
$Qe=e$.

 Lemma~\ref{lem:Q-existence} establishes a correspondence between generalized skew
cospectral oriented graphs via a regular rational orthogonal matrix $Q$.
Moreover, when the oriented graph is controllable, the corresponding matrix $Q$ is unique.
\begin{lemma}[\cite{JohnsonNewman,QiuWangWang,WangXu}]\label{lem:Q-existence}
Let $\Sigma$ be a controllable oriented graph and let $\Delta$ be an oriented graph. Then $\Sigma$ and $\Delta$ have the same generalized skew spectrum if and only if there exists a unique regular rational orthogonal matrix $Q$ such that
\begin{equation}
Q^{\rm T}S(\Sigma)Q=S(\Delta).
\end{equation}
Moreover,
\begin{equation}\label{eq:walk-intertwine}
 Q^{\mathrm T}W(\Sigma)=W(\Delta).
\end{equation}
\end{lemma}

Let $\mathcal C(\Sigma)$ denote the set of all oriented graphs having the same
generalized skew spectrum as $\Sigma$, and let $\mathrm{RO}_n(\mathbb{Q})$
denote the set of all regular rational orthogonal matrices of order $n$.
Define
\[
\mathcal Q(\Sigma)
:=
\bigcup_{\Delta\in\mathcal C(\Sigma)}
\left\{
Q\in\mathrm{RO}_n(\mathbb Q):
Q^{\mathrm T}S(\Sigma)Q=S(\Delta)
\right\}.
\]

\begin{lemma}[\cite{QiuWangWang}]
Let $\Sigma$ be a controllable oriented graph. Then $\Sigma$ is $\mathrm{DGSS}$ if and only if every matrix in
$\mathcal{Q}(\Sigma)$ is a permutation matrix.
\end{lemma}

The \emph{level} of a regular rational orthogonal matrix $Q$, denoted by
$\ell(Q)$ or simply $\ell$, is the smallest positive integer $\ell$ such that $\ell Q$ is an
integral matrix. Clearly, a regular rational orthogonal matrix $Q$ is a
permutation matrix if and only if $\ell(Q)=1$. Consequently, the preceding
lemma can be equivalently reformulated as follows.

\begin{corollary}[\cite{QiuWangWang}]\label{l=1}
Let $\Sigma$ be a controllable oriented graph with
$\det W(\Sigma)\neq 0$. Then $\Sigma$ is $\mathrm{DGSS}$ if and only if every
matrix $Q\in\mathcal{Q}(\Sigma)$ has level $1$.
\end{corollary}

The next lemma shows the relationship between the level $\ell$ and the last invariant factors of $W(\Sigma)$.

\begin{lemma}[\cite{ref6}]\label{l dn}

    Let $X$ be a nonsingular integral matrix with last invariant factor $d_n(X)$. If $Q$ is rational and $Q^{\mathrm T}X$ is integral, then the level of $Q$ divides $d_n(X)$. In particular, if $Q\in\mathcal Q(\Sigma)$, then
$
 \ell(Q)\mid d_n(W(\Sigma)).$
\end{lemma}

\subsection{The $p$-main polynomial} \label{The $p$-main polynomial}
The $p$-main polynomial, introduced in~\cite{WangWangZhu} in the study of generalized spectral determination of controllable simple graphs, plays a central role in our proof.
\begin{definition}
    Let $p$ be a prime. Let $\Sigma$ be an oriented graph with skew-adjacency matrix $S$. The \emph{$p$-main polynomial} of $\Sigma$, denoted by $f_p(\Sigma;x)$, is the monic polynomial of smallest degree in $\F_p[x]$ such that $
 f_p(\Sigma;S)e=0.$
\end{definition}

\begin{lemma}[\cite{QiuWangWang}]\label{lem:first-r-columns}
Let $p$ be a prime and let $r=\rank_p W(\Sigma)$. Then the first $r$ columns of $W(\Sigma)$ consist of a basis of the column space of $W(\Sigma)$.
\end{lemma}

\begin{lemma}\label{lem:p-main-degree}
For every prime $p$, $
 \deg f_p(\Sigma;x)=\rank_p W(\Sigma).$
\end{lemma}
\begin{proof}
Let $r=\rank_p W(\Sigma)$. By Lemma~\ref{lem:first-r-columns},
the vectors $e,Se,\ldots,S^{r-1}e$ are linearly independent, while
$S^re\in\operatorname{Span}\{e,Se,\ldots,S^{r-1}e\}$. Hence there exists a
monic polynomial of degree $r$ that annihilates $e$. The minimality of $r$
implies that $
\deg f_p(\Sigma;x)=r=\rank_p W(\Sigma).$
\end{proof}

The following theorem provides a unified criterion for excluding a prime divisor
from the level of $Q$.  The theorem is inspired by the arguments developed in \cite{QiuWangWang,WangWangZhu}.

\begin{theorem}[Prime-exclusion theorem]\label{lem:prime-exclusion}  Let $\Sigma\in \mathcal{F}_{n}$ and $\Delta$ be generalized skew cospectral with $\Sigma$.
 Let $Q\in\mathcal Q(\Sigma)$ have level $\ell(Q)$ such that $Q^{\rm T}S(\Sigma)Q=S(\Delta)$.  Let $p$ be a prime divisor of $d_n(W(\Sigma))$. Assume that
 \begin{enumerate}
    \item $p^2\nmid d_n(W(\Sigma))$;
    \item $f_p(\Sigma;x)=f_p(\Delta;x)$.
 \end{enumerate}
Then $p\nmid\ell(Q)$.
\end{theorem}

\begin{proof}
    Write
$A=S(\Sigma)$, $B=Q^{\mathrm T}AQ=S(\Delta)$, and let $r={\rm rank}_{p}W\,(\Sigma)$. Let $f(x)\in \mathbb{Z}[x]$ be  a monic  polynomial such that $f(x)\equiv f_p(\Sigma;x)\equiv f_p(\Delta;x) \pmod{p}$,
where $f(A)e \equiv  f(B)e \equiv 0 \:(\mathrm{mod}\;p)$. Note that $\nullity_pW(\Sigma)=n-r$. It follows from Lemma~\ref{lem:p-main-degree} that the polynomial $f(x)$ has degree $r$ and $\nullity_pW(\Sigma)=\nullity_pW(\Delta)$.
 Now we define two new matrices as follows:
\begin{equation}
    \bar {W}(\Sigma)=[e,Ae,\dots,A^{r-1}e,\frac{f(A)e}{p},\frac{Af(A)e}{p},\dots,\frac{A^{n-r-2}f(A)e}{p}, \frac{A^{n-r-1}f(A)e}{p}],
\end{equation}

\begin{equation}
    \bar {W}(\Delta)=[e,Be,\dots,B^{r-1}e,\frac{f(B)e}{p}   ,\frac{Bf(B)e}{p}    ,\dots,\frac{B^{n-r-2}f(B)e}{p},\frac{B^{n-r-1}f(B)e}{p}].
\end{equation}
Both $\bar {W}(\Sigma)$ and $\bar {W}(\Delta)$ are integral matrices. It follows from Lemma~\ref{lem:Q-existence} that $Q^{\rm T}\bar {W}(\Sigma)=\bar {W}(\Delta)$. By Lemma~\ref{l dn}, we obtain $
\ell(Q)\mid d_n(\bar W(\Sigma)),$
and hence
$
\ell(Q)\mid\det\bar W(\Sigma).
$

Because every invariant factor of $W(\Sigma)$ divides $d_n(W(\Sigma))$ and
$p^2\nmid d_n(W(\Sigma))$, each invariant factor contains at most one factor $p$. Moreover, as $\nullity_pW(\Sigma)=n-r$, then
exactly $n-r$ invariant factors are divisible by $p$. Thus, we have $v_p(\det W(\Sigma))=n-r$.  Note that $\det\bar {W}(\Sigma)=p^{r-n}\det W(\Sigma)$, which implies that $p \nmid \det\bar{W}(\Sigma)$ and thus $p\nmid \ell(Q)$.

This completes the proof.
\end{proof}

\subsection{Invariant subspaces and polynomial decomposition}
Throughout this section, let $p$ be a fixed odd prime. Recall that
$\mathbb{F}_p$ denotes the finite field of order $p$, and
$\overline{\mathbb{F}}_p$ its algebraic closure. We begin with some basic
notions concerning orthogonal and dual spaces.

For $u,v\in\mathbb{F}_p^n$, we say that $u$ and $v$ are \emph{orthogonal},
denoted by $u\perp v$, if $
u^{\rm T}v=0.$
Similarly, two subspaces $U,V\subseteq\mathbb{F}_p^n$ are said to be
\emph{orthogonal}, denoted by $U\perp V$, if $
u^{\rm T}v=0 $
for all $u\in U$ and $v\in V$.

\begin{definition}[\cite{L.Babai}]\label{V com}
Let $V$ be a subspace of $\mathbb{F}_p^n$. The orthogonal space
of $V$ is defined by
\[
V^\perp
=
\bigl\{
u\in\mathbb{F}_p^n:
v^{\rm T}u=0 \text{ for every } v\in V
\bigr\}.
\]
\end{definition}

\begin{definition}[\cite{HoffmanKunze1971}]
Let $V$ be a vector space over a field $\mathbb{F}$. A \emph{linear
functional} on $V$ is a linear map from $V$ to $\mathbb{F}$. The
\emph{dual space} of $V$, denoted by $V^*$, is the vector space of all
linear functionals on $V$; equivalently,
\[
V^*=\operatorname{Hom}_{\mathbb{F}}(V,\mathbb{F}).
\]
If $V$ is finite-dimensional, then $\dim V^*=\dim V$.
\end{definition}
The standard bilinear form $\langle u,v\rangle=u^{\rm T}v$ on
$\mathbb{F}_p^n$ is nondegenerate. Consequently, for every subspace
$V\subseteq\mathbb{F}_p^n$,
\[
\dim V^\perp=n-\dim V
\qquad\text{and}\qquad
(V^\perp)^\perp=V.
\]
Unlike the Euclidean inner product, the standard bilinear form on
$\mathbb{F}_p^n$ need not be positive definite. Thus one may have $
V\cap V^\perp\neq\{0\},$
and consequently $
\mathbb{F}_p^n=V\oplus V^\perp$
does not hold in general.

Therefore, the characteristic polynomial of a linear operator cannot in general be factorized directly in terms of its restrictions to $V$ and $V^\perp$.  Nevertheless, when $V$ is invariant under the operator, such a factorization can be obtained by considering the induced operator on the quotient space together with the transpose operator restricted to $V^\perp$. We now establish this factorization.

\begin{lemma}\label{fpoly}
Let $S$ be a matrix over $\mathbb{F}_p$, and let $U$ be an
$S$-invariant subspace of $\mathbb{F}_p^n$. Then $U^\perp$ is
$S^{\rm T}$-invariant, and
\[
\chi(S;x)
=
\chi(S|_U;x)\,
\chi(S^{\rm T}|_{U^\perp};x).
\]
\end{lemma}

\begin{proof}
Let $\langle x,y\rangle=x^{\rm T}y$ denote the standard bilinear form on
$\mathbb{F}_p^n$. For all $x,y\in\mathbb{F}_p^n$, we have
$\langle Sx,y\rangle=\langle x,S^{\rm T}y\rangle$.

We first show that $U^\perp$ is $S^{\rm T}$-invariant. Let $y\in U^\perp$.
Since $U$ is $S$-invariant, we have $Sx\in U$ for every $x\in U$. Hence
$\langle x,S^{\rm T}y\rangle=\langle Sx,y\rangle=0$ for every $x\in U$,
and thus $S^{\rm T}y\in U^\perp$.

Since $U$ is $S$-invariant, the matrix $S$ induces a well-defined linear
operator on the quotient space $\mathbb{F}_p^n/U$, given by
\[
\widetilde{S}\colon
\mathbb{F}_p^n/U\longrightarrow\mathbb{F}_p^n/U,
\qquad
\widetilde{S}(v+U)=Sv+U.
\]
Choose a basis of $U$ and extend it to a basis of $\mathbb{F}_p^n$. With
respect to this basis, the matrix of $S$ has the block upper triangular form
\[
S=
\begin{pmatrix}
S|_U & *\\
0 & \widetilde{S}
\end{pmatrix}.
\]
Thus
$\chi(S;x)=\chi(S|_U;x)\chi(\widetilde S;x)$.

It remains to determine the characteristic polynomial of
$\widetilde{S}$. Define
\[
\Phi\colon
\mathbb{F}_p^n/U\longrightarrow (U^\perp)^*,
\qquad
\Phi(v+U)(w)=\langle v,w\rangle
\quad\text{for }w\in U^\perp.
\]

We first verify that $\Phi$ is well-defined. Suppose that
$v+U=v'+U$. Then $v'-v\in U$, and hence
$\langle v'-v,w\rangle=0$ for every $w\in U^\perp$. Therefore,
$\Phi(v+U)=\Phi(v'+U)$. Moreover, if $\Phi(v+U)=0$, then we have
$\langle v,w\rangle=0$ for every $w\in U^\perp$, which indicates that
$v\in(U^\perp)^\perp=U$. Therefore, $v+U$ is the zero element
of $\mathbb{F}_p^n/U$, and hence we get that $\Phi$ is injective.

On the other hand,
\[
\dim(\mathbb{F}_p^n/U)
=
n-\dim U
=
\dim U^\perp
=
\dim (U^\perp)^*.
\]
It follows that $\Phi$ is an isomorphism.

We now compare the action of \(\widetilde S\) under this isomorphism with the
dual action on \((U^\perp)^*\). Let $T=S^{\rm T}|_{U^\perp}$. For \(v\in \mathbb{F}_p^n\) and
\(w\in U^\perp\), we have
\[
\begin{aligned}
\Phi\bigl(\widetilde{S}(v+U)\bigr)(w)
&=\Phi(Sv+U)(w)\\
&=\langle Sv,w\rangle\\
&=\langle v,S^{\rm T}w\rangle\\
&=\Phi(v+U)(Tw)\\
&=T^*\bigl(\Phi(v+U)\bigr)(w).
\end{aligned}
\]
Hence $\Phi\widetilde S=T^*\Phi$, so $\widetilde S$ is similar to $T^*$.
Since a linear operator and its dual have the same characteristic polynomial, we obtain
\[
\chi(\widetilde{S};x)
=
\chi(T^*;x)
=
\chi(T;x)
=
\chi(S^{\rm T}|_{U^\perp};x).
\]
Therefore,
\[
\chi(S;x)
=
\chi(S|_U;x)\,
\chi(S^{\rm T}|_{U^\perp};x).
\]
This completes the proof.
\end{proof}

\begin{corollary}\label{factorization}
   Let $S$ be a skew-symmetric matrix over $\mathbb{F}_p$ and let $U$ be an $S$-invariant subspace of $\mathbb{F}_p^n$. Then
\begin{equation}\label{eq.factorization}
   \chi(S;x) = \chi(S|_U;x)\, \chi(-S|_{U^\perp};x).
\end{equation}

\end{corollary}
\begin{proof}
Since $S$ is skew-symmetric over $\mathbb{F}_p$, we have
$S^{\rm T}=-S$. Therefore, it follows from Lemma~\ref{fpoly} that $
\chi(S;x)
=
\chi(S|_U;x)\,
\chi(-S|_{U^\perp};x)$.
\end{proof}

\begin{corollary}\label{factorization sy}
   Let $S$ be a symmetric matrix over $\mathbb{F}_p$ and let $U$ be an $S$-invariant subspace of $\mathbb{F}_p^n$. Then
\begin{equation}
   \chi(S;x) = \chi(S|_U;x)\, \chi(S|_{U^\perp};x).
\end{equation}

\end{corollary}
\begin{proof}
Since $S$ is symmetric over $\mathbb{F}_p$, we have
$S^{\rm T}=S$. The assertion follows.
\end{proof}

A factorization equivalent to Corollary~\ref{factorization sy} was
obtained by Wang et al.~\cite{WangWangZhu} for symmetric matrices.
Lemma~\ref{fpoly} extends this factorization to arbitrary matrices,
while Corollary~\ref{factorization} gives the corresponding
specialization to skew-symmetric matrices. This approach also yields
a shorter derivation of the symmetric case.
\section{Proof of Theorem~\ref{thm:main}}\label{proof thm}
This section is devoted to the proof of Theorem~\ref{thm:main}. By the strategy developed in Sections~\ref{rational orthogonal matrices} and~\ref{The $p$-main polynomial}, it suffices to exclude every prime divisor of $d_n(W(\Sigma))$ from the level of each matrix in $\mathcal Q(\Sigma)$. The key step is to show that generalized skew-cospectral oriented graphs share the same $p$-main polynomial under the assumptions of the main theorem.

\subsection{$p=2$}\label{PQ}
We first deal with the prime $2$.
\begin{lemma}[cf.~\cite{Wang2017}]\label{lem: Me=0}
    Let $M$ be a symmetric matrix over $\F_2$. If $M^2=0$, then $Me=0$.
\end{lemma}
\begin{proof}
    For each $i$,
\begin{center}
    $(M^2)_{ii}=\sum_j m_{ij}m_{ji}
 =\sum_j m_{ij}^2
 =\sum_j m_{ij}=(Me)_i. $
\end{center}
Hence $M^2=0$ implies $Me=0$.
\end{proof}

Suppose that the characteristic polynomial of $\Sigma$ is $
\varphi(x)=x^{n}+c_{1}x^{n-1}+\cdots+c_{n-1}x+c_{n}.$
Define $\psi_{\Sigma}(x)\in\mathbb{F}_2[x]$ by
\begin{equation}\label{eq:psi-2}
\psi_{\Sigma}(x):=
\begin{cases}
x^m+c_2x^{m-1}+c_4x^{m-2}+\cdots+c_{2m},
& n=2m,\\[2mm]
x^{m+1}+c_2x^m+c_4x^{m-1}+\cdots+c_{2m}x,
& n=2m+1.
\end{cases}
\end{equation}

\begin{lemma}[cf.~\cite{Wang2017}]\label{lem:psi-even}
Let $\Sigma\in \mathcal{F}_{n}$ with skew-adjacency matrix $S$. Then $\psi_{\Sigma}(S)e=0$ over $\F_2$.
\end{lemma}
\begin{proof}
Modulo $2$, the matrix $S$ is symmetric, and hence every polynomial in $S$ is
symmetric.  Since the odd-indexed coefficients of $\chi(S;x)$ vanish,
by Eq.~\eqref{eq:psi-2}, we have
\[
\psi_\Sigma(x)^2=
\begin{cases}
\chi(S;x), & n=2m,\\
x\chi(S;x), & n=2m+1.
\end{cases}
\]
By the Cayley--Hamilton theorem, it follows in both cases that
$\psi_{\Sigma}(S)^2=0$ over $\mathbb{F}_2$. Since $\psi_{\Sigma}(S)$ is symmetric,
it follows from Lemma~\ref{lem: Me=0} that
$\psi_{\Sigma}(S)e=0$.
\end{proof}

\begin{proposition}\label{prop:exclude-two}
Let $\Sigma\in\mathcal{F}_n$ satisfy the assumptions of
Theorem~\ref{thm:main}. Then every $Q\in\mathcal Q(\Sigma)$ has odd level.
\end{proposition}

\begin{proof}
Let $r=\rank_2W(\Sigma)=\lceil n/2\rceil$ and
$k=n-r=\lfloor n/2\rfloor$. Fix $Q\in\mathcal Q(\Sigma)$, and let $\Delta\in\mathcal C(\Sigma)$ be such that $
Q^{\rm T}S(\Sigma)Q=S(\Delta).$  Since $d_n(W(\Sigma))$ is square-free,
exactly $k$ invariant factors of $W(\Sigma)$ are even, and hence
$v_2(\det W(\Sigma))=k$.

 Since
$\Sigma$ and $\Delta$ have the same characteristic polynomial, it implies that
$\psi_\Sigma(x)=\psi_\Delta(x)$. By Lemma~\ref{lem:psi-even},
$\psi_\Sigma(S(\Sigma))e=0$ over $\mathbb F_2$. Since
$\deg\psi_\Sigma=r=\deg f_2(\Sigma;x)$ by
Lemma~\ref{lem:p-main-degree}, then we have
$f_2(\Sigma;x)=\psi_\Sigma(x)$.

Similarly, Lemma~\ref{lem:psi-even} gives
$\psi_\Delta(S(\Delta))e=0$, and hence
$f_2(\Delta;x)\mid\psi_\Delta(x)$. Therefore, we obtain
$\rank_2W(\Delta)=\deg f_2(\Delta;x)\le r$, which indicates
$\nullity_2W(\Delta)\ge k$. On the other hand, it follows from Eq.~\eqref{eq:walk-intertwine} that
$|\det W(\Delta)|=|\det W(\Sigma)|$, which implies that
$v_2(\det W(\Delta))=k$, and thus we have
$\nullity_2W(\Delta)\le v_2(\det W(\Delta))=k$. Therefore,
$\nullity_2W(\Delta)=k$ and $\rank_2W(\Delta)=r$.
By Lemma~\ref{lem:p-main-degree},
$\deg f_2(\Delta;x)=r$. Since
$f_2(\Delta;x)\mid\psi_\Delta(x)$ and both are monic of degree $r$, then we get
\[
f_2(\Delta;x)=\psi_\Delta(x)=\psi_\Sigma(x)=f_2(\Sigma;x).
\]

Thus, the assumptions of Theorem~\ref{lem:prime-exclusion} are satisfied for
$p=2$, and consequently $2\nmid\ell(Q)$.
\end{proof}

\subsection{$p$ is odd}

Throughout this section, $p$ is a fixed odd prime. All polynomial factorizations are taken over $\F_p$, while eigenvectors may be taken over the algebraic closure $\overline{\F}_p$. We simply write $K=\Ker_p(W^{\rm T})$.
For a monic polynomial $f\in\F_p[x]$ of degree $r$, define
\[
 f^\star(x)=(-1)^r f(-x).
\]
Then $f^\star$ is monic and $(f^\star)^\star=f$.

We now return to the invariant polynomial $\Phi_p(\Sigma;x)$. The next observation records the symmetry of its roots.

\begin{lemma}\label{lem:root-symmetry}
If $\lambda$ is a root of
$\Phi_p(\Sigma;x)$ over $\overline{\mathbb{F}}_p$, then $-\lambda$ is also a root of
$\Phi_p(\Sigma;x)$.
\end{lemma}

\begin{proof}
Let $e$ denote the all-one column vector, and write $J=ee^{\mathrm T}$. Note that
\[
\det(A-uv^{\mathrm T})
=
\det(A)-v^{\mathrm T}\operatorname{adj}(A)u.
\]
Define $h(x)=e^{\mathrm T}\operatorname{adj}(xI-S)e.$ Then
\begin{equation}\label{eq:rank-one-identity}
\begin{aligned}
\chi(S+J;x)
&=\det\bigl(xI-(S+J)\bigr)\\
&=\det\bigl((xI-S)-ee^{\mathrm T}\bigr)\\
&=\chi(S;x)-h(x).
\end{aligned}
\end{equation}
Since $S^{\mathrm T}=-S$, we have
\begin{equation}\label{eq:chi-parity}
\begin{aligned}
\chi(S;-x)
&=\det(-xI-S)\\
&=(-1)^n\det(xI+S)\\
&=(-1)^n\det\bigl((xI-S)^{\mathrm T}\bigr)\\
&=(-1)^n\chi(S;x).
\end{aligned}
\end{equation}
Similarly, as
$
\operatorname{adj}(cA)
=
c^{n-1}\operatorname{adj}(A)$
and
$
\operatorname{adj}(A^{\mathrm T})
=
\operatorname{adj}(A)^{\mathrm T},$
we obtain
\begin{equation}\label{eq:h-parity}
\begin{aligned}
h(-x)
&=e^{\mathrm T}\operatorname{adj}(-xI-S)e\\
&=(-1)^{n-1}
  e^{\mathrm T}
  \operatorname{adj}\bigl((xI-S)^{\mathrm T}\bigr)e\\
&=(-1)^{n-1}
  e^{\mathrm T}
  \operatorname{adj}(xI-S)^{\mathrm T}e\\
&=(-1)^{n-1}
  e^{\mathrm T}
  \operatorname{adj}(xI-S)e\\
&=(-1)^{n-1}h(x).
\end{aligned}
\end{equation}
The penultimate equality follows from the fact that
$e^{\mathrm T}\operatorname{adj}(xI-S)e$ is a scalar and is therefore
equal to its transpose.

Since $\lambda$ is a root of
$\Phi_p(\Sigma;x)$ over $\overline{\mathbb{F}}_p$, the definition of $\Phi_p(\Sigma;x)$ implies that $\lambda$ is a common root of $\chi(S;x)$ and
$\chi(S+J;x)$. Hence, we have $\chi(S;\lambda)=0$ and
$\chi(S+J;\lambda)=0.$
It follows from Eq.~\eqref{eq:rank-one-identity} that \[
h(\lambda)
=
\chi(S;\lambda)-\chi(S+J;\lambda)
=
0.\]
By Eqs.~\eqref{eq:chi-parity} and ~\eqref{eq:h-parity},
we have
\[
\chi(S+J;-\lambda)
=
\chi(S;-\lambda)-h(-\lambda)
=
0.
\]
Thus, $-\lambda$ is also a common root of $\chi(S;x)$ and
$\chi(S+J;x)$, and hence it is a root of
$\Phi_p(\Sigma;x)$.
\end{proof}

\begin{lemma}\label{factor-star}
Under the condition~\textup{(ii)} of Theorem~\ref{thm:main}, if $\sfp(\Phi_p(\Sigma;x))=\varphi_1(x)\cdots\varphi_s(x)$ is its irreducible factorization, then  $\varphi_i^{\star}=\varphi_i$ for $1\le i\le s$. Moreover, $\sfp(\Phi_p(\Sigma;x))^{\star}=\sfp(\Phi_p(\Sigma;x))$.
\end{lemma}

\begin{proof}
Fix $i$ and let $\lambda$ be a root of $\varphi_i$. By Lemma~\ref{lem:root-symmetry}, we know that  $-\lambda$ is a root of some factor $\varphi_j$. If $j\neq i$, then $\lambda$ is a common root of $\varphi_i(x)$ and $\varphi_j(-x)$, contradicting Theorem~\ref{thm:main} condition~\textup{(ii)}. Hence $j=i$. It follows that $\varphi_i$ and $\varphi_i^{\star}$ share a common root. Both are monic irreducible polynomials of the same degree, so $\varphi_i^\star=\varphi_i$. Multiplying over $i$ gives the final assertion.
\end{proof}

\begin{lemma}\label{WS}
For any $t\in\mathbb{F}_p$, the subspace
$K$ is invariant under $S+tJ$.
\end{lemma}

\begin{proof}
Let $\xi\in K$. Then
$e^{\rm T}(S^{\rm T})^k\xi=0$ for $k=0,1,\ldots,n-1$.
Since $S^{\rm T}=-S$, it follows that
$e^{\rm T}S^k\xi=0$ for $k=0,1,\ldots,n-1$. Let
$\chi(S;x)=x^n+a_{n-1}x^{n-1}+\cdots+a_1x+a_0$.
By the Cayley--Hamilton theorem, we have
\[
S^n=-\bigl(a_{n-1}S^{n-1}+\cdots+a_1S+a_0I\bigr),
\]
and hence $e^{\rm T}S^n\xi=0$. Therefore,
$e^{\rm T}(S^{\rm T})^kS\xi=0$ for $k=0,1,\ldots,n-1$, which gives
$W^{\rm T}S\xi=0$.

Moreover, $e^{\rm T}\xi=0$, and hence $J\xi=ee^{\rm T}\xi=0$. Thus
\[
W^{\rm T}(S+tJ)\xi
=
W^{\rm T}S\xi+tW^{\rm T}J\xi
=
0.
\]
Therefore, $(S+tJ)\xi\in K$, and the result follows.
\end{proof}

\begin{lemma}\label{divides}

  We have $\chi(S|_K;x)$ divides $\Phi_p(\Sigma;x)$. Moreover,
$\sfp(\Phi_p(\Sigma;x))$ divides $\chi(S|_K;x)$ whenever either condition~\textup{(ii)} of Theorem~\ref{thm:main} holds or $
\chi(S|_K;x)^\star=\chi(S|_K;x).$
\end{lemma}
\begin{proof}
Recall $K=\operatorname{Ker}_p(W^{\rm T}).$
By Lemma~\ref{WS}, the space $K$ is invariant under both $S$ and $S+J$. Since any vector
$v\in K$ satisfies $e^{\rm T}v=0$, we have $J|_K=0$. It follows that
$
\chi\left(S|_K;x\right)
=
\chi\left((S+J)|_K;x\right).$
Consequently,
\[
\chi\left(S|_K;x\right)
\mid
\gcd\bigl(\chi(S;x),\chi(S+J;x)\bigr)
=
\Phi_p(\Sigma;x).
\]

It remains to prove the second assertion. Write $
\operatorname{sfp}(\Phi_p(\Sigma;x))
=
\varphi_1(x)\varphi_2(x)\cdots \varphi_s(x),$
where the $\varphi_i$ are distinct monic irreducible polynomials over $\mathbb{F}_p[x]$. Fix $i$, and let $\lambda$ be a root of
$\varphi_i$ over $\overline{\mathbb{F}}_p$. Then $\lambda$ is a common eigenvalue of $S$ and $S+J$, and by Lemma~\ref{lem:root-symmetry}, $-\lambda$ is also an eigenvalue of $S$. Choose nonzero vectors $\eta,\xi$ such that $
 S\eta=-\lambda\eta,$ and
 $(S+J)\xi=\lambda\xi,$ respectively.
Then
\[
0
=
\xi^{\rm T}(\lambda I+S)\eta
=
\eta^{\rm T}(\lambda I-S)\xi
=
\eta^{\rm T}J\xi
=
(e^{\rm T}\eta)(e^{\rm T}\xi).\]
Hence either $e^{\rm T}\eta=0$ or $e^{\rm T}\xi=0$.

If $e^{\mathrm T}\xi=0$, then $J\xi=0$, and hence  $S\xi=\lambda\xi$. Moreover,
as $e^{\mathrm T}(S^{\mathrm T})^k\xi=(-\lambda)^k e^{\mathrm T}\xi=0,$ for every $k$, then we have $\xi\in K$. Thus, $\lambda$ is a root of $\chi(S|_K;x)$.

If $e^{\mathrm T}\eta=0$, then similarly $\eta\in K$, so $-\lambda$ is a root of $\chi(S|_K;x)$. Under condition~\textup{(ii)}, by Lemma~\ref{factor-star}, we know that $-\lambda$ is also a root of the same irreducible factor $\varphi_i$. Alternatively, if $\chi(S|_K;x)^\star=\chi(S|_K;x)$, then $\lambda$ is a root of $\chi(S|_K;x)$ whenever $-\lambda$ is.

Therefore, in either case, $\varphi_i(x)$ and
$\chi\left(S|_K;x\right)$ have a common root in
$\overline{\mathbb{F}}_p$. Since $\varphi_i(x)$ is irreducible over
$\mathbb{F}_p$, we conclude that $\varphi_i(x)\mid \chi\left(S|_K;x\right).$
As $i$ was arbitrary, we obtain \[
\operatorname{sfp}(\Phi_p(\Sigma;x))
\mid
\chi\left(S|_K;x\right).\]
This completes the proof.
\end{proof}


\begin{corollary}\label{co compare}
	Under condition~\textup{(ii)} of Theorem~\ref{thm:main},  $\deg\operatorname{sfp}(\Phi_p(\Sigma; x)) \leq \operatorname{nullity}_p W(\Sigma) \leq \deg \Phi_p(\Sigma; x)$.
\end{corollary}
\begin{proof}
    Note that $\deg\chi(S|_K;x)=\dim K=\operatorname{nullity}_p W(\Sigma)$. The assertion clearly follows from Lemma~\ref{divides}.
\end{proof}


\begin{lemma}\label{main po}
   We have $f_p(\Sigma; x)=\chi(S|_{K^{\bot}};x)$.
\end{lemma}
\begin{proof}
    By Lemma~\ref{WS}, $K$ is $S$-invariant. Since $S^{\rm T}=-S$, Lemma~\ref{fpoly} implies that $K^\perp$ is also $S$-invariant.

    Let $g(x)=\chi(S|_{K^{\bot}};x)$. Note that $\dim K^\perp=n-\nullity_pW(\Sigma)=\rank_pW(\Sigma)$. By Lemma~\ref{lem:p-main-degree}, we know that  $
 \deg f_p(\Sigma;x)=\deg g(x).$  Moreover, $e\in K^\perp$. By the Cayley--Hamilton theorem applied to $S|_{K^\perp}$, we have $g(S)e=0$. Hence, by the minimality of the $p$-main polynomial, we know $f_p(\Sigma;x)\mid g(x).$
Since the two monic polynomials have the same degree, they are equal.
\end{proof}
\begin{lemma}\label{mp}
    Let $p$ be an odd prime. Suppose that $
\deg\sfp(\Phi_p(\Sigma;x))=\nullity_pW(\Sigma)$
and that condition~\textup{(ii)} of Theorem~\ref{thm:main} holds for $p$.
Then
\[
f_p(\Sigma;x)=\frac{\chi(S;x)}{\sfp(\Phi_p(\Sigma;x))}.
\]
\end{lemma}
\begin{proof}
  By the degree--nullity assumption, then $
\deg\operatorname{sfp}(\Phi_p(\Sigma;x))
=
\nullity_pW
=
\dim K.$
Since $\operatorname{sfp}(\Phi_p(\Sigma;x))$ divides
$\chi(S|_K;x)$, then we have $ \mathrm{sfp}(\Phi_p(\Sigma;x))=\chi\left(S|_K; x\right)$.  Lemma~\ref{factor-star} indicates
$\chi(S|_K;x)^\star=\chi(S|_K;x)$, while
Eq.~\eqref{eq:chi-parity} gives
$\chi(S;x)^\star=\chi(S;x)$. Corollary~\ref{factorization} yields
\[
\chi(S;x)=\chi(S|_K;x)\,\chi(-S|_{K^\perp};x).
\]
Applying the star operation gives
\[
\chi(S;x)=\chi(S|_K;x)\,\chi(-S|_{K^\perp};x)^\star.
\]
 Thus, we obtain $\chi(-S|_{K^\perp};x)=\chi(-S|_{K^\perp};x)^{\star}=\chi(S|_{K^\perp};x)$.
Therefore, \[\chi(S;x) = \chi(S|_K;x)\, \chi(S|_{K^\perp};x).\]
Finally, Lemma~\ref{main po} gives
$ f_p(\Sigma;x)=\chi(S|_{K^\perp};x)$, and hence
\[
f_p(\Sigma;x)=\frac{\chi(S;x)}{\sfp(\Phi_p(\Sigma;x))}.
\]
\end{proof}

   The next consequence is the key generalized cospectral invariance needed in the proof of the main theorem.

\begin{corollary}\label{co invariant}
Let $\Sigma\in\mathcal F_n$ satisfy the assumptions of Theorem~\ref{thm:main}, and let $\Delta$ be generalized skew cospectral with $\Sigma$. If $p$ is an odd prime divisor of $d_n(W(\Sigma))$, then $\nullity_pW(\Delta)=\nullity_pW(\Sigma)$ and $f_p(\Delta; x)=f_p(\Sigma; x).$
\end{corollary}
\begin{proof}
    Let $k=\operatorname{nullity}_p W(\Sigma)$. Since $\Sigma$ and $\Delta$ are generalized
skew cospectral, we have
$\Phi_p(\Delta;x)=\Phi_p(\Sigma;x).$ Hence condition~\textup{(ii)} for the fixed prime $p$ is inherited by $\Delta$. Using condition~\textup{(i)} for $\Sigma$ and Corollary~\ref{co compare}
applied to $\Delta$, we obtain
    \[k=\deg\operatorname{sfp}(\Phi_p(\Sigma; x))=\deg\operatorname{sfp}(\Phi_p(\Delta; x)) \leq \operatorname{nullity}_p W(\Delta).\]

 On the other hand, since $d_n(W(\Sigma))$ is square-free, exactly $k$ invariant factors of $W(\Sigma)$ are
divisible by $p$. Thus
$v_p(\det W(\Sigma))=k$.
 It follows from Lemma \ref{lem:Q-existence} that $|\det W(\Delta)|=|\det W(\Sigma)|$,  which implies that
$v_p(\det W(\Delta))=k$. Thus, we know that $\operatorname{nullity}_p W(\Delta)\le k$.
Therefore, one indicates that \[\deg\operatorname{sfp}(\Phi_p(\Delta; x))=k=\operatorname{nullity}_p W(\Delta),\] and $\nullity_pW(\Delta)=\nullity_pW(\Sigma)$.

Therefore, we obtain that  $\Delta$ satisfies all conditions of Theorem~\ref{thm:main}. By Lemma~\ref{mp}, we have
\[
f_p(\Delta;x)
=
\frac{\chi(S(\Delta);x)}{\sfp(\Phi_p(\Delta;x))}
=
\frac{\chi(S(\Sigma);x)}{\sfp(\Phi_p(\Sigma;x))}
=
f_p(\Sigma;x).
\]
 The assertion follows.
   \end{proof}

Next, we exclude each odd prime satisfying the hypotheses of Theorem~\ref{thm:main} from the level of a matrix in $\mathcal{Q}(\Sigma)$.

\begin{proposition}\label{prop:exclude-p}
Let $\Sigma\in\mathcal F_n$ satisfy the assumptions of Theorem~\ref{thm:main}, let $Q\in\mathcal Q(\Sigma)$ have level $\ell(Q)$, and let $p$ be an odd prime divisor of $d_n(W(\Sigma))$. Then $p\nmid\ell(Q)$.
\end{proposition}

\begin{proof}
Let $Q\in\mathcal Q(\Sigma)$, and let $\Delta$ be generalized skew cospectral with $\Sigma$ satisfying $Q^{\rm T}S(\Sigma)Q=S(\Delta)$. Thus, Corollary~\ref{co invariant} gives $
f_p(\Sigma;x)=f_p(\Delta;x).$
Since $d_n$ is square-free,
$p^2\nmid d_n(W(\Sigma))$. By  Prime-exclusion theorem~\ref{lem:prime-exclusion},  we obtain $p \nmid \ell(Q)$.
\end{proof}
We are now in a position to prove Theorem~\ref{thm:main}.

\begin{proof}[Proof of Theorem~\ref{thm:main}]
Let $Q\in\mathcal{Q}(\Sigma)$ and let $\ell=\ell(Q)$. By Lemma~\ref{l dn},
one has $\ell\mid d_n(W)$. Proposition~\ref{prop:exclude-two} shows that $\ell$ is odd. On the contrary, suppose that $\ell>1$, and let $p$ be an odd prime divisor of $\ell$, then $p\mid d_n$. Therefore, the assumptions of the Theorem~\ref{thm:main}  apply to $p$, and Proposition~\ref{prop:exclude-p} gives $p\nmid\ell$; a contradiction. Thus $\ell=1$ for every $Q\in\mathcal{Q}(\Sigma)$. Corollary~\ref{l=1} implies that $\Sigma$ is $\mathrm{DGSS}$.
\end{proof}

\begin{corollary}\label{cor:self-converse}
If $\Sigma\in \mathcal{F}_n$ satisfies the assumptions of Theorem~\ref{thm:main}, then $\Sigma$ is self-converse.
\end{corollary}

\begin{proof}
The converse $\Sigma^{\rm T}$ always has the same generalized skew spectrum as $\Sigma$. By Theorem~\ref{thm:main}, it must be isomorphic to $\Sigma$.
\end{proof}

\begin{lemma}\label{lem:self-converse-kernel}
Suppose that $\Sigma\in\mathcal{F}_n$ is self-converse. Then $
\chi(S|_K;x)^\star=\chi(S|_K;x).$
\end{lemma}

\begin{proof}
Since $\Sigma$ is self-converse, then there exists a permutation matrix $P$ such that $
P^{\rm T}SP=-S.$
Then $Pe=e$ and $SP=-PS$. If $v\in K$, then
$e^{\mathrm T}S^j v=0$ for every $j\ge0$. Hence
\[
 e^{\mathrm T}S^jPv=(-1)^je^{\mathrm T}PS^jv=(-1)^je^{\mathrm T}S^jv=0,
\]
so $Pv\in K$. Thus $K$ is $P$- invariant. Restricting
$P^{\mathrm T}SP=-S$ to $K$ shows that $S|_K$ is similar to $-S|_K$, which indicates $
\chi(S|_K;x)^\star=\chi(S|_K;x).$
\end{proof}

\begin{corollary}\label{cor:nullity-one}
Let $\Sigma\in\mathcal{F}_n$ be self-converse and let $p$ be an odd prime. If $
\nullity_pW(\Sigma)=1,$
then $
\deg\sfp(\Phi_p(\Sigma;x))=1.$
Consequently, conditions~\textup{(i)} and~\textup{(ii)} of Theorem~\ref{thm:main} are automatically satisfied for any $p$.
\end{corollary}

\begin{proof}
Since $\Sigma$ is self-converse, by Lemma~\ref{lem:self-converse-kernel}, we have $\chi(S|_K;x)^\star=\chi(S|_K;x)$. Hence, Lemma~\ref{divides} gives $\sfp(\Phi_p(\Sigma;x))\mid\chi(S|_K;x)$ and $\Phi_p(\Sigma;x)$ is nonconstant. Since $\dim K=\nullity_p W(\Sigma)=1$, it follows that $\deg\sfp(\Phi_p(\Sigma;x))=\nullity_p W(\Sigma)=1$. Therefore, condition~\textup{(i)} holds, while condition~\textup{(ii)} is automatic by Remark~\ref{rem:vacuous}.
\end{proof}
\begin{remark}
    Theorem~\ref{thm:old} is recovered as a special case of Theorem~\ref{thm:main}.  Indeed, under the assumptions of Theorem~\ref{thm:old}, the Smith normal form of $W(\Sigma)$ has the form
\[
\diag(
\underbrace{1,\ldots,1}_{\left\lceil n/2\right\rceil},
\underbrace{2,\ldots,2,2b}_{\left\lfloor n/2\right\rfloor}),
\]
where $b$ is an odd square-free integer; see \cite{QiuWangWang}.
Thus $d_n=2b$ is square-free, and $\nullity_pW=1$ for every odd prime divisor $p$ of $b$. Since self-converse is already assumed in Theorem~\ref{thm:old}, Corollary~\ref{cor:nullity-one} verifies conditions~\textup{(i)} and \textup{(ii)} of Theorem~\ref{thm:main}.
\end{remark}

\section{Some examples}\label{sec:examples}
In this section, we give some examples for illustrations.

\begin{example}
Let $\Sigma$ be an oriented graph of order $n=10$. The skew-adjacency matrix $S(\Sigma)$ of
$\Sigma$ is given as follows:
\[
S(\Sigma)=\begin{pmatrix}
0 & 0 & 0 & 1 & -1 & 1 & 1 & 0 & -1 & 0 \\
0 & 0 & 1 & 0 & -1 & 1 & 0 & -1 & 0 & 1 \\
0 & -1 & 0 & -1 & 0 & 0 & -1 & 0 & -1 & 1 \\
-1 & 0 & 1 & 0 & 1 & -1 & 0 & -1 & 1 & 0 \\
1 & 1 & 0 & -1 & 0 & 0 & 1 & 1 & 0 & -1 \\
-1 & -1 & 0 & 1 & 0 & 0 & 0 & 0 & -1 & 1 \\
-1 & 0 & 1 & 0 & -1 & 0 & 0 & -1 & 0 & 1 \\
0 & 1 & 0 & 1 & -1 & 0 & 1 & 0 & 1 & 0 \\
1 & 0 & 1 & -1 & 0 & 1 & 0 & -1 & 0 & -1 \\
0 & -1 & -1 & 0 & 1 & -1 & -1 & 0 & 1 & 0
\end{pmatrix}.
\]
It can be computed that the $\mathrm{SNF}$ of $W(\Sigma)$ is
       \[
           \mathrm{diag} \left ( 1,1,1,1,1,2,2,2,2\times7\times13\times743,2\times7\times13\times743 \right ).
       \]
Thus, $
 d_{10}=2\times7\times13\times743$
is square-free and $\rank_2W=5$.  By calculating, we know that
\[
\begin{array}{c|c}
p & \sfp(\Phi_p(\Sigma;x)) \\ \hline
7   & x^2+2 \\
13  & x^2+7 \\
743 & x^2+342
\end{array}
\]
 It is clear that for odd prime $p$ of $d_n$, we have $\deg\big(\operatorname{sfp}(\Phi_p(\Sigma;x))\big) = \operatorname{nullity}_p W(\Sigma)=2$. Moreover, each quadratic polynomial in the table is irreducible over the corresponding field.  Therefore, it follows from Theorem~\ref{thm:main} that the $\Sigma$ is $\mathrm{DGSS}$.

 On the other hand, Theorem~\ref{thm:old} fails to determine whether this example is $\mathrm{DGSS}$, because $
 2^{-5}|\det W(\Sigma)|=(7\times13\times743)^2$
is not square-free.
\end{example}
The following example shows that condition~\textup{(ii)} cannot in general be omitted from Theorem~\ref{thm:main}.
\begin{example}\label{exam 2}
    Let $\Sigma$ be an oriented graph of order $n = 7$. The skew-adjacency matrix $S(\Sigma)$ of
$\Sigma$ is given as follows:
    \[
S(\Sigma)=\begin{pmatrix}
0 & -1 & 1 & -1 & 0 & -1 & 1 \\
1 & 0 & 0 & 0 & 0 & 0 & 1 \\
-1 & 0 & 0 & -1 & 0 & 1 & 0 \\
1 & 0 & 1 & 0 & 1 & -1 & -1 \\
0 & 0 & 0 & -1 & 0 & 0 & -1 \\
1 & 0 & -1 & 1 & 0 & 0 & 0 \\
-1 & -1 & 0 & 1 & 1 & 0 & 0
\end{pmatrix}.
\]

It can be computed that the $\mathrm{SNF}$ of $W(\Sigma)$ is
       \[
           \mathrm{diag} \left ( 1,1,1,1,2,2\times5\times23,2\times5\times23 \right ).
       \]
       Hence $d_7=2\times5\times23$ is square-free and $\rank_2W=4$.  For $p=23$, we have $\sfp(\Phi_{23}(\Sigma;x))=x^2-11$ and $\nullity_{23}W(\Sigma)=2$. Thus, condition~\textup{(i)} holds. Moreover, since $x^2-11$ is
irreducible over $\mathbb{F}_{23}$, condition~\textup{(ii)} is
automatically satisfied. For $p=5$, $\operatorname{sfp}(\Phi_p(\Sigma; x))=(1+x) (4+x)$ and $\nullity_5W(\Sigma)=2$. Hence, condition~\textup{(i)} holds, whereas
condition~\textup{(ii)} fails. Consequently, among the hypotheses of Theorem~\ref{thm:main}, only condition~\textup{(ii)} for $p=5$ fails.

Consider
\[
Q_1=\frac{1}{5}\begin{pmatrix}
1 & 1 & 3 & 3 & -2 & -1 & 0 \\
3 & 3 & -1 & -1 & -1 & 2 & 0 \\
-3 & 2 & 1 & 1 & 1 & 3 & 0 \\
1 & 1 & 3 & -2 & 3 & -1 & 0 \\
2 & -3 & 1 & 1 & 1 & 3 & 0 \\
0 & 0 & 0 & 0 & 0 & 0 & 5 \\
1 & 1 & -2 & 3 & 3 & -1 & 0
\end{pmatrix}.
\]
A direct computation shows that $
Q_1^{\rm T}Q_1=I,$
$Q_1e=e$, and
$\ell(Q_1)=5.$
Moreover,
\[
S(\Delta)=Q_1^{\rm T}S(\Sigma)Q_1=
\begin{pmatrix}
0 & 0 & 1 & 1 & 0 & 0 & -1 \\
0 & 0 & 0 & 1 & 1 & 0 & 0 \\
-1 & 0 & 0 & 1 & -1 & 1 & -1 \\
-1 & -1 & -1 & 0 & 1 & 0 & 0 \\
0 & -1 & 1 & -1 & 0 & 1 & 0 \\
0 & 0 & -1 & 0 & -1 & 0 & 1 \\
1 & 0 & 1 & 0 & 0 & -1 & 0
\end{pmatrix}.
\]
Thus $\Delta$ is generalized skew cospectral with $\Sigma$. The out-degree sequences of $\Sigma$ and $\Delta$ are $
(0,1,2,2,2,2,3)$ and
$(1,1,2,2,2,2,2),$  respectively,
so the two oriented graphs are not isomorphic. Hence, $\Sigma$ is not $\mathrm{DGSS}$.

This indicates that condition~\textup{(ii)} cannot in general be omitted from the stated criterion, even if all the remaining assumptions of Theorem~\ref{thm:main} are satisfied.
\end{example}
As observed in \cite{WangWangZhu}, simple graphs satisfying the condition of Theorem~\ref{wang} are determined by the $\mathrm{SNF}$ of their walk matrices. For oriented graphs, by contrast, no analogous result holds. By computing $\Delta$ in Example~\ref{exam 2}, we find that the $\mathrm{SNF}$ of $W(\Delta)$ is
\[ \mathrm{diag} \left ( 1,1,1,1,2,2\times5\times23,2\times5\times23 \right ).\]
However, $\Sigma$ is not isomorphic to $\Delta$.

\section{Conclusion and Future Work}\label{sec:conclusion}
In this paper, we established a new sufficient condition for a controllable oriented graph to be determined by its generalized skew spectrum when the last invariant factor of its skew-walk matrix is square-free. By treating the prime $2$ and the odd primes separately, we showed that generalized skew-cospectral oriented graphs have the same $p$-main polynomial under the conditions of Theorem~\ref{thm:main}, which allows every prime divisor of the level of a regular rational orthogonal matrix to be excluded. The resulting criterion extends the previous square-free determinant criterion to cases with higher $p$-nullity, as illustrated by the examples.

At the end of this paper, we mention a natural problem suggested by our computational experiments. In all examples we have examined, the square-freeness of the last invariant factor $d_n$ of $W(\Sigma)$ seems to imply $
\rank_2 W(\Sigma)=\left\lceil\frac n2\right\rceil.$
This leads to the following question.
\begin{problem}
Let $\Sigma\in \mathcal{F}_n$, and let $d_n$ be
the last invariant factor of $W(\Sigma)$. Suppose that $d_n$ is square-free, then $
 \rank_2 W(\Sigma)=\left\lceil\frac n2\right\rceil.$
\end{problem}
An affirmative answer would make the \(\mathbb{F}_2\)-rank condition for \(W(\Sigma)\) in Theorem~\ref{thm:main} redundant and reveal a closer connection between the Smith normal form of $W(\Sigma)$ and its rank over $\mathbb{F}_2$. We leave this problem for future study.

\end{document}